\documentclass[11pt]{article}

\usepackage{latexsym}
\usepackage{amssymb,amsmath}
\usepackage{amsthm}
\usepackage{epsfig}
\usepackage{amsfonts}
\usepackage{graphicx}
\usepackage{color}

\date{}

\newtheorem{theorem}{Theorem}[section]

\begin{document}

\title{\bf  
Dominated and dominator colorings over (edge) corona and hierarchical products}

\author{Sandi Klav\v zar $^{a,b,c}$ \and Mostafa Tavakoli $^{d,}$ \footnote{Corresponding author.}} 

\date{}

\maketitle

\begin{center}
$^a$ Faculty of Mathematics and Physics, University of Ljubljana, Slovenia\\
{\tt sandi.klavzar@fmf.uni-lj.si}

\medskip

$^b$ Faculty of Natural Sciences and Mathematics, University of Maribor, Slovenia\\
\medskip

$^c$ Institute of Mathematics, Physics and Mechanics, Ljubljana, Slovenia\\
\medskip

$^d$ Department of Applied Mathematics, Faculty of Mathematical Sciences,\\
Ferdowsi University of Mashhad, P.O.\ Box 1159, Mashhad 91775, Iran\\
{\tt m$\_$tavakoli@um.ac.ir}

\end{center}

\begin{abstract}
Dominator coloring of a graph is a proper (vertex) coloring with the property that every vertex is either alone in its color class or adjacent to all vertices of at least one color class. A dominated coloring of a graph is a proper coloring such that every color class is dominated with at least one vertex. The dominator chromatic number of corona products and of edge corona products is determined. Sharp lower and upper bounds are given for the dominated chromatic number of edge corona products. The dominator chromatic number of hierarchical products is bounded from above and the dominated chromatic number of hierarchical products with two factors determined. An application of dominated colorings in genetic networks is also proposed.
 

\end{abstract}

\section{Introduction}

Graph colorings form one of the most investigated areas of graph theory. This is in particular so because colorings of (vertices of) graphs form  natural models for a vast number of practical problems involving facility location problems in operational research. It often happens that besides the requirement that adjacent vertices receive different colors, some additional condition(s) on a coloring must be fulfilled. In this way new variants of colorings appear, a relative recent and interesting variant is the following.

A {\it dominator coloring} of a graph $G$ is a proper vertex coloring with the additional property that every vertex $u$ of $G$ forms a color class, or $u$ is adjacent to all vertices of at least one color class. The smallest number of colors needed for a dominator coloring of $G$ is the {\it dominator chromatic number} $\chi_d(G)$ of $G$. This concept was studied for the first by Gera, Horton, and Rasmussen~\cite{Gera}, several papers followed afterwards. Chellali and Maffray~\cite{chellali-2012} proved, among other results, a very interesting fact that determining whether $\chi_d(G)\le 3$ holds can be accomplished in polynomial time. Moreover, they showed that the dominator chromatic number of $P_4$-free graphs can also be computed in polynomial time. Gera~\cite{gera-2007} proved that if $T$ is a nontrivial tree, then $\chi_d(G)\in \{\gamma(T) + 1, \gamma(T) + 2\}$, and later Boumediene Merouane and Chellali~\cite{d1} characterized trees T attaining each of the possibilities.  The dominator chromatic number of Cartesian products of $P_2$ and $P_3$ by arbitrary paths and cycles was determined in~\cite{chen-2018, chen-2017}. Some additional Cartesian products, several direct products, and some corona products were studied in~\cite{Paul}. Dominator colorings of {M}ycielskian graphs were investigated in~\cite{abid-2019}. 

A concept closely related to dominator colorings is the following. A proper coloring of a graph $G$ is a {\it dominated coloring} if each color class is dominated by at least one vertex, that is, for each color class there exists a vertex that is adjacent to all the vertices of the class. The minimum number of colors needed for a dominated coloring of $G$ is the {\it dominated chromatic number} $\chi_{dom}(G)$ of $G$. This concept was introduced in 2015 (the paper being submitted in 2012 though) by Boumediene Merouane et al.~\cite{dom1} where they adopted algorithmic approach for this problem and proved that if $G$ is triangle-free, then $\chi_{dom}(G)$ equals the total domination number of $G$. In~\cite{bagan-2017} different variants of colorings (including dominator and dominated ones) were compared mostly from the algorithmic point of view and very many interesting results presented. Let us just emphasize the dichotomy asserting that dominated coloring is polynomial on claw-free graphs while the dominator coloring is NP-complete on claw-free graphs. This dichotomy indicated that although dominator colorings and dominated colorings appear quite similar, they are in fact strikingly different.   

We proceed as follows. In the rest of the introduction we first discuss applications of dominated colorings and propose a new application in genetic networks. At the end of the introduction standard definitions needed in this paper are listed. Then, in Section~\ref{sec:corona}, we determine the dominator chromatic number of corona products. It is significantly different from the dominated chromatic number which was earlier determined in~\cite{TJM}. In Section~\ref{sec:edge-corona} we first determine the dominator chromatic number of edge corona products. For the dominated chromatic number of such graphs we give sharp lower and upper bonds. We get an equality in particular for edge corona products in which the first factor is bipartite with minimum degree at least $2$. In the final section we bound from above the dominator chromatic number of hierarchical products and determine the dominated chromatic number for the case of two factors. 

\subsection{Applications of dominated colorings}

Already in 2014, Chen~\cite{d3} provided an application of dominated coloring in social networks for finding the minimum stranger groups who can become friends later by an intermediary.
We now propose another applicability in genetic networks as follows. 

In a genetic interaction network $G$, genes (proteins) are represented as vertices (nodes) and their relationships as
edges. Some genes (proteins) do not have direct interactions with each other, but they may be under
regulation by a common gene(protein). Actually, the common gene (protein) can regulate the function of the other genes (proteins), see \cite{gen1,gen2}. 
Therefore, the dominated coloring is to find the minimum groups of genes (proteins) in the genetic (protein) interaction network with two below properties: 
\begin{enumerate}
\item genes (proteins) in the same group do not have direct interactions with each other,
\item genes (proteins) in the same group are regulated by a common gene (protein).
\end{enumerate}

\subsection{Some definitions}

If $G$ is a graph we will denote its order with  $n(G)$ and its size with $m(G)$. For a positive integer $n$, we will use the notation $[n] = \{1,\ldots,n\}$. The chromatic number of $G$ is of course denoted with $\chi(G)$. In a (proper) $k-$coloring of $G$, a {\em color class} is the set of vertices assigned the same color.  If $c:V(G) \rightarrow [k]$ is (proper) coloring of $G$ and $i\in [k]$, then let
$$C_G(i) = \{u\in V(G):\ c(u) = i \}$$
be the {\em color class $i$}. If $G$ will be clear from the context, we will abbreviate its notation to $C(i)$. 

A {\it matching} in a graph $G$ is a set of nonadjacent edges of $G$. The {\it matching number} $\alpha'(G)$ is the cardinality of a largest matching in $G$. If $M$ is a matching, then a vertex is {\em $M$-matched} (or just {\em matched}) if it is an endpoint of an edge from $M$. The {\em vertex cover number} $\beta(G)$ of $G$ is the cardinality of a smallest set of vertices such that each edge has at least one endpoint in the set. 

\section{Coloring corona products}
\label{sec:corona}

The {\em corona product} $G\circ H$ of graphs $G$ and $H$ is obtained from one copy of $G$ and $n(G)$ copies of $H$ by joining with an edge each vertex of the $i^{\rm th}$ copy of $H$, $i\in [n(G)]$, to the $i^{\rm th}$ vertex of $G$, cf.~\cite{Yeh}. If $g\in V(G)$, then the copy of $H$ in $G\circ H$ corresponding to $g$ with be denoted with $H_g$. We may consider the vertex set of $G\circ H$ to be 
$$V(G\circ H) = V(G) \bigcup_{g\in V(G)} V(H_g)\,.$$ 
The dominated chromatic number of corona products is already known.
 
\begin{theorem} {\rm \cite[Theorem 4.4]{TJM}}
\label{thm:corona-dom}
If $G$ and $H$ are graphs, then $\chi_{dom}(G\circ H )= n(G)\chi(H)$.
 \end{theorem}

\begin{figure}[ht!]
 \centerline{\includegraphics[scale=.27]{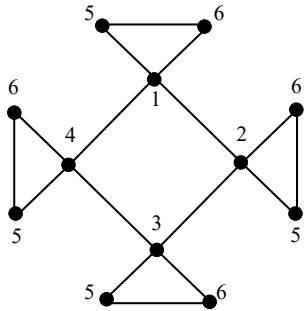}}
\caption{\label{fig1} A dominator coloring of $C_4 \circ K_2$.}
\end{figure}

A dominator coloring of the corona product $C_4 \circ K_2$ is shown in Fig~\ref{fig1}. The dominator chromatic number of $K_n\circ K_1$ has been reported in~\cite{Paul}. We now give a general result for the dominator chromatic number of corona products. 

\begin{theorem}
\label{thm:corona-d}
If $G$ and $H$ are graphs, then $\chi_d(G\circ H) = n(G) + \chi(H)$.
\end{theorem}

\begin{proof}
Set $n = n(G)$ and define  $c: V(G\circ H) \rightarrow [n + \chi(H)]$ as follows. First, color the vertices of $V(G)$ with pairwise different colors from $[n]$. Second, for each $g\in V(G)$, let $c$ restricted to $H_g$ be a $\chi(H)$-coloring of $H$ using colors from the set $\{n+1, \ldots, n+\chi(H)\}$, see Fig.~\ref{fig1} again. Then $c$ is a dominator coloring of $G\circ H$. Indeed, each vertex $g\in V(G)$ forms a color class of cardinality $1$, while each vertex from $H_g$ is adjacent to the color class $\{g\}$. Therefore, $\chi_d(G\circ H)\leq n+\chi(H)$.

It remains to prove that $\chi_d(G\circ H)\geq n+\chi(H)$. Let  $c$ be an arbitrary dominator coloring of $G\circ H$ and suppose that $c(g) = c(g')$ for vertices $g,g'\in V(G) \subseteq V(G\circ H)$. Then we claim that there exists a color class that lies completely in $V(H_g)$. Let $u\in V(H_g)$ and suppose that $c(u) = s$. If $C(s) = \{u\}$, there is nothing to prove. Otherwise, $|C(s)| \ge 2$ and hence $u$ must dominate a color class $r$, where $r\ne s$. Note that $r\ne c(g)$ because $ug'\notin E(G\circ H)$. But then $C(r)\subseteq V(H_g)$, proving the claim. If $C(s) = \{u\}$ then define a coloring $c'$ of $G\circ H$ by setting  
$$c'(x) = \left\{\begin{array}{ll}
s;      & x = g\,, \\
c(g); & x = u\,,\\
c(x); & {\rm otherwise}\,,
\end{array}\right.
$$
otherwise, that is, if $|C(s)| \ge 2$, define $c'$ with  
$$c'(x) = \left\{\begin{array}{ll}
r;      & x = g\,, \\
c(g); & x\in V(H_g), c(x) = r\,,\\
c(x); & {\rm otherwise}\,.
\end{array}\right.
$$
Note that in either of the two cases, $c'$ is a dominator coloring of $G\circ H$ that uses the same number of colors as $c$. Moreover, $c'$ uses one more color on $V(G)$ as $c$. Repeating this construction as long as necessary, we arrive at a dominator coloring $c''$ of $G\circ H$ that uses the same number of colors as $c$, and such that if $g, g'\in V(G)$, $g\ne g'$, then $c''(g) \ne c''(g')$. 

In the rest we may without loss of generality assume that if $g\in V(G)$, then $c''(g)\in [n]$.  Let $g\in V(G)$ be the vertex with $c''(g) = 1$. If $c''$ restricted to $H_g$ uses only colors bigger than $n$, then clearly $c''$ uses at least $n + \chi(H)$ colors. Suppose next that $c''$ restricted to $H_g$ uses some color $i\in [n]$ and let $g'\in V(G)$ be the vertex with $c''(g') = i$. Clearly, $i\ne 1$. We claim that $c''$ restricted to $H_{g'}$ uses a color that is used only in $H_{g'}$. For this sake let $u$ be an arbitrary vertex of $H_{g'}$. We have nothing to prove if $C''(u) = \{u\}$.  Otherwise, no matter whether $c''(u)$ appears on some other vertex of $H_{g'}$ or elsewhere, the vertex $x$ must be adjacent to all the vertices of a color class that lies completely in $H_{g'}$. It follows that the color of this color class is used only in $H_{g'}$, proving the claim. Hence each color used in $H_g$ is either bigger than $n$ or leads to its private new color bigger than $n$. Therefore, $c''$ uses at least $n + \chi(H)$ colors, hence also $c$ uses at least $n + \chi(H)$ colors. As $c$ was an arbitrary dominator coloring of $G\circ H$ we conclude that  $\chi_d(G\circ H)\geq n+\chi(H)$. 
\end{proof}
 
Note that Theorems~\ref{thm:corona-dom} and~\ref{thm:corona-d} reveal that  $\chi_{dom}$ and $\chi_d$ behave strikingly differently on corona products. Roughly speaking, $\chi_{dom}$ is a quadratic, while $\chi_d$ is a linear invariant. In particular, ff $\Gamma$ is a bipartite graph and $W$ is a connected graph of order $k+2$, then   $\chi_{dom}(W\circ \Gamma)-\chi_d(W\circ \Gamma)=2(k+2)-(k+2)-2=k$. Hence for each $k\geq 0$ there exists a graph $G$ with $\chi_{dom}(G)-\chi_d(G)=k$.
 
\section{Coloring edge corona products}
\label{sec:edge-corona}

The {\em edge corona} $G\diamond H$ of graphs $G$ and $H$ is obtained by taking one copy of $G$ and $m(G)$ disjoint copies of $H$ one-to-one assigned to the edges of $G$, and for every edge $gg'\in E(G)$ joining $g$ and $g'$ to every vertex of the copy of $H$ associated to $gg'$, see~\cite{e-cor1,e-cor2}. If $gg'\in E(G)$, then the copy of $H$ in $G\diamond H$ corresponding to $e=gg'$ will be denoted with $H_{gg'}$ (or simply $H_e$). Hence we may consider the vertex set of $G\diamond H$ to be 
$$V(G\diamond H) = V(G) \bigcup_{gg'\in E(G)} V(H_{gg'})\,.$$ 
The edge corona $C_4\diamond K_2$ is shown in Fig.~\ref{fig2} along with its dominator coloring.  

\begin{figure}[ht!]
 \centerline{\includegraphics[scale=.27]{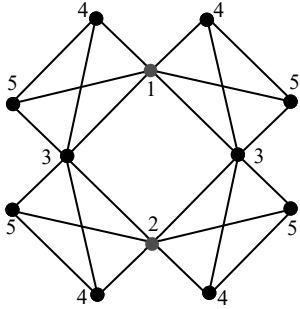}}
\caption{\label{fig2} A dominator coloring of $C_4\diamond K_2$.}
\end{figure}

\begin{theorem}
If $G$ and $H$ are graphs, then $\chi_d(G\diamond H)=\beta(G) + \chi(H) + 1$.
\end{theorem}
 
\begin{proof}
Let $K$ be a minimum vertex cover of $G$, so that $|K| = \beta(G)$. 

We first prove that $\chi_d(G\diamond H)\leq \beta(G) + \chi(H) + 1$. To reach this aim, consider the following coloring $c:V(G\diamond H) \rightarrow [\beta(G) + \chi(H) + 1]$:
\begin{itemize}
\item color vertices of $K$ injectively with colors $1, \ldots, \beta(G)$; 
\item color vertices from $V(G)-K$ with color $\beta(G)+1$;
\item for each $e\in E(G)$, color $H_{e}$ with colors $\{\beta(G)+2, \ldots, \beta(G)+\chi(H)+1\}$.
\end{itemize}
An example of such a coloring is given in Fig.~\ref{fig2}. The coloring $c$ is a dominator coloring of $G\diamond H$. Indeed, each vertex $g\in K$ forms a color class of cardinality one. Consider next now a vertex $u$ from some $H_{gg'}$. As $K$ is a vertex cover, we may without loss of generality assume that $g\in K$. But then $u$ is adjacent to the color class $\{g\}$.

It remains to prove that $\chi_d(G\diamond H)\geq \beta(G) + \chi(H) + 1$. Let $c$ be an arbitrary dominator coloring of $G\diamond H$ and suppose that there exists $gg'\in E(G)$ such that $|C(c(g))|>1$ and  $|C(c(g'))|>1$. In this case, we claim that there must exist a color class $r$ that lies completely in $H_{gg'}$. Let $u\in V(H_{gg'})$ and suppose that $c(u)=s$. If $C(s)=\{u\}$, there is nothing to prove. Otherwise, $|C(s)|>1$ and hence $u$ must dominate a color class $r$, where $r\neq s$.
Since $c(g)\neq r$ and $c(g') \neq r$ we see that $C(r)\subseteq V(H_{gg'})$. 
 By changing the colors of certain vertices, we construct from $c$ another dominator coloring $c'$ of $G\diamond H$ as follows:
 $$c'(x) = \left\{\begin{array}{ll}
r;      & x = g\,, \\
c(g); & x\in E(H_{gg'}), c(x) = r\,,\\
c(x); & {\rm otherwise}\,.
\end{array}\right.
$$
Indeed, $c'$ is a dominator coloring because now $u$ still dominates the color class $r$ (which consists of a single element). The coloring $c'$ uses the same number of colors as $c$. We use this technique of recoloring to reach a dominator coloring $c''$ of $G\diamond H$ with this property that for each $gg'\in E(G)$ at least one of $|C''(c''(g))|=1$ and $|C''(c''(g'))|=1$ holds.

Set $K=\{g\in V(G):\ |C''(c''(g))|=1\}$. Because of the above property of $c''$ for each edge, $K$ is a vertex cover of $G$ and consequently $|K|\geq \beta(G)$. 
Clearly, $c''$ does not use colors that are used on $K$ for coloring of other vertices. Suppose first that $|K|=n(G)$. Then $c''$ uses at least $n(G)+\chi(H)$ colors for coloring $G\diamond H$. Since $n(G)\geq \beta(G)+1$ it follows that $c''$ uses at least $\beta(G) + \chi(H) + 1$ colors. Suppose second that $|K|<n(G)$. Then each proper coloring of the vertices from $V(G\diamond H)\setminus K$ uses  at least $\chi(H)+1$ colors. Indeed, if $g\notin K$ and $e$ is an arbitrary edge having $g$ as one of its endpoints, then the join of $g$ and $H_e$ is a subgraph of $G\diamond H$ which needs at least $\chi(H)+1$ colors. Hence also in this case $c''$ uses at least $\beta(G) + \chi(H) + 1$ colors.
\end{proof}

\begin{figure}[ht!]
 \centerline{\includegraphics[scale=.29]{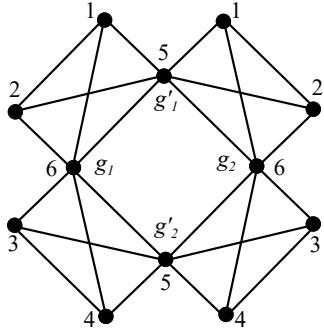}}
\caption{\label{fig5} A dominated coloring of $C_4\diamond  K_2$.}
\end{figure}

\begin{theorem}\label{th4}
If $G$ is a graph without pendant vertices, then $$\chi_{dom}(G\diamond H)\geq \alpha'(G)\chi(H)+\chi_{dom}(G).$$
\end{theorem}

\begin{proof}
Set $r = \alpha'(G)$ and let $M=\{g_1g'_1, \ldots, g_{r}g'_{r}\}$ be a maximum matching of $G$. 
Our proof has three steps.
First, we observe that we need at least $r\,\chi(H)$ colors for coloring all the vertices from $H_{g_ig_i'}$, $i\in [r]$.
Second, we show that we need at least $r\,\chi(H)$ colors for coloring the vertices of all $H_e$ where $e\notin M$. (We apply $r\,\chi(H)$ colors used in the first step for coloring these copies).
Third, we prove that the colors used in the previous steps cannot be assigned to the vertices of $G$.

The fact that we need at least $r\,\chi(H)$ colors for coloring all the vertices from $H_{g_ig_i'}$, $i\in [r]$, follows from the assumption that the edges $g_ig_i'$ form a matching and hence a vertex from $H_{g_jg_j'}$ and a vertex from $H_{g_kg_k'}$, where $k\ne k'$, have no common neighbor. For the second step of our proof consider an edge $g_ig_i'$ and a neighbor of $g_i$ different from $g_i'$, say $g$. (Such a neighbor exists since we have assumed that $G$ has no pendant vertices.) Let $X_i$ denote the set of colors used on $H_{g_ig_i'}$ which are also used in $H_{g_ig}$, that is, denoting the corresponding coloring with $c$ we set 
$$X_i=\{c(v)\ : \ v\in V(H_{g_ig_i'})\}\cap \{c(v)\ : \ v\in V(H_{g_ig})\}.$$
Similarly, suppose that $X'_i$ denotes the set of colors of $H_{g_ig_i'}$ which $c$ uses on $H_{g'_ig'}$ where $g'\neq g_i$.
So a dominated coloring of $G\diamond H$ can use some colors of $H_{g_ig_i'}$ for coloring the vertices of $H_{g_ig}$ where $g$ (that $g\neq g_i'$) is a neighbour of $g_i$ which forms $X_i$, and use remaining colors of $H_{g_ig_i'}$ for coloring the vertices of $H_{g'_ig'}$
where $g'$ (that $g'\neq g_i$) is a neighbour of $g_i'$ which forms $X'_i$.
(For more illustration, see Fig.~\ref{fig5}. In this figure, $M=\{g_1g_1', g_2g_2'\}$ is a maximum matching of $C_4$, $\{1,2\}$ is the set of colors used in $H_{g_1g_1'}$,
$\{3,4\}$ is the set of colors used in $H_{g_2g_2'}$, $X_1=\emptyset$, $X_1'=\{1,2\}$, $X_2=\emptyset$, $X_2'=\{3,4\}$.)
Thus, $|X_i\cup X_i'|\leq \chi(H)$ and $|X_i\cap X_i'|=0$, because if there exists $k\in (X_i\cap X_i')$, then the color class $C(k)$ would not be dominated by a vertex. 
Also, since $G$ does not have pendant vertices, then $|E(G)\setminus M|\geq r$. 

Since $M$ is a maximum matching, an edge $e\in E(G)\setminus M$ is either adjacent to two members of $M$, say $e=g_i'g_j$, or $e$ is adjacent to one member of $M$, say $e=g_ig$. In the first case, vertices of $H_{g_i'g_j}$ are colored with colors of $X_i'\cup X_j$, and so $|X_i'\cup X_j|\geq \chi(H)$. In the second case, vertices of $H_{g_ig}$ are colored with colors of $X_i$, and so $|X_i|\geq \chi(H)$.  Therefore, at least $r\,\chi(H)$ colors are needed for coloring the vertices of $H_e$'s in $G\diamond H$, where $e\notin M$. 

To complete our proof, it is sufficient to show that the colors of $\bigcup_{i=1}^{r}(X_i\cup X_i')$ cannot be used in vertices of $G$.
If $gg_i\in E(G)$ and $g\neq g'_i$, then $c(g)\notin X_i$, and (since $g$ is adjacent to all vertices of $H_{g_ig}$) $c(g)\notin X_i'$.
Therefore, each coloring of $G\diamond H$ needs at least $r\chi(H)$ colors for coloring of copies of $H$ that cannot be applied for vertices of $G$. We conclude that $\chi_{dom}(G\diamond H)\geq \alpha'(G)\chi(H)+\chi_{dom}(G)$.
\end{proof}

Consider $C_4\diamond K_2$ depicted in Fig. \ref{fig5}. $M=\{g_1g_1', g_2g_2'\}$ is a maximum matching of $C_4$ and so $\alpha'(C_4)=2$. Then, by Theorem \ref{th4},  $$\chi_{dom}(C_4\diamond K_2)\geq \alpha'(C_4)\chi(K_2)+\chi_{dom}(C_4)=2\times 2+2=6.$$
On the other hand, the coloring from Fig.~\ref{fig5} demonstrates that $\chi_{dom}(C_4\diamond K_2) \le 6$, hence the bound of Theorem~\ref{th4} is sharp.

\begin{theorem} \label{th6}
If $G$ has $k$ pendant vertices, then $\chi_{dom}(G\diamond H)\geq \alpha'(G)\chi(H)+k$.
\end{theorem}

\begin{proof}
Set $r = \alpha'(G)$ and let $M=\{g_1g_1',\ldots, g_{r}g_{r}'\}$ be a maximum matching of $G$. 
As in the proof of Theorem~\ref{th4} we infer that at least $r\,\chi(H)$ colors are required in a dominated coloring $c$ for the vertices from  $H_{g_ig_i'}$, $i\in [r]$. Let $g$ be a pendant vertex of $G$. If $g$ is an end-point of an edge from $M$, then $c(g)$ is different from all the colors used on $H_{g_ig_i'}$, $i\in [r]$. Otherwise, having in mind that $M$ is a maximum matching, $g$ is adjacent to a matched vertex, say $g_i$. But then $g_i$ requires an additional color. Hence each of the pendant vertices adds one more color to $c$. 
\end{proof}

Theorem~\ref{th6} implies that $\chi_{dom}(P_4\diamond K_4)\geq \alpha'(P_4)\chi(K_4)+k=2\times 4+2=10$. On the other hand, it is not difficult to find a dominated coloring of $P_4\diamond K_4$ using $10$ colors. Hence also the bound of Theorem~\ref{th6} is sharp.

In Theorem~\ref{th4} and~\ref{th6} we have bounded $\chi_{dom}(G\diamond H)$ from below using the matching number of $G$. In our next result we bound $\chi_{dom}(G\diamond H)$ from above using the vertex cover number of $G$.

\begin{theorem}
If $G$ and $H$ are graphs, then $\chi_{dom}(G\diamond H)\leqslant\chi_{dom}(G)+\beta(G)\chi(H)$, with equality when $G$ is bipartite graph without pendant vertices.
\end{theorem}

\begin{proof}
Set $\beta = \beta(G)$ and let $K = \{v_1, \ldots, v_{\beta}\}$ be a vertex cover of $G$. Partition $E(G)$ into subsets of edges $E_1, \ldots, E_{\beta(G)}$, such that if $e\in E_i$, then $v_i$ is an endpoint of $e$. It is clear that such a partition always exists since $K$ is a vertex cover. 

Let $c$ be a coloring of $G\diamond H$ defined as follows. First, for each set of edges $E_i$ reserve private $\chi(H)$ colors and color with then each of the subgraphs $H_e$, $e\in E_i$. Second, color the vertices of $G$ with additional $\chi_{dom}(G)$ colors. (See Fig.~\ref{fig5} for an example of such a coloring. In this figure, $K=\{g_1',g_2'\}$.) Thus $c$ is a coloring using $\chi_{dom}(G)+\beta(G)\chi(H)$ colors. Moreover, $c$ is a dominated coloring because each color class on $G$ is dominated by a vertex from $G$, while the other color classes are dominated by appropriate vertices from $K$. Hence $\chi_{dom}(G\diamond H)\leq\chi_{dom}(G)+\beta(G)\chi(H)$.

Now, suppose $G$ is a bipartite graph without pendant vertices. Then, by Theorem \ref{th4}, $\chi_{dom}(G\diamond H)\geq \alpha'(G)\chi(H)+\chi_{dom}(G)$. Recall that the famous K\"onig–Egerv\'ary Theorem asserts that if $G$ is a bipartite graph, then $\alpha'(G) = \beta(G)$.  Therefore, $\chi_{dom}(G\diamond H)= \beta(G)\chi(H)+\chi_{dom}(G)$.
\end{proof}

\section{Hierarchical products}
\label{sec:hierarchical}

Suppose $\{G_i = (V_i, E_i) \}_{i=1}^N$, is a family of graphs having a distinguished or root vertex $r$. Following Barri\'ere et al.~\cite{b1,b2}, the hierarchical product
 $H = G_N \sqcap \ldots\sqcap G_2\sqcap G_1$ is the graph with vertices as 
 $N$-tuples $(x_N,\ldots, x_1)$, $x_i\in V_i$, and edges defined as follows:
{\small $$ (x_N, \ldots ,x_3,x_2,x_1) \sim \left\lbrace
\begin{array}{c l}
(x_N, \ldots x_3,x_2,y_1); &
\text{$y_1\sim x_1$\: in $\:G_1$,} \\
(x_N, \ldots, x_3,y_2,x_1); & \text{$y_2\sim x_2$\: in $\:G_2$ and
$x_1=r$,}\\
(x_N, \ldots, y_3,x_2,x_1); & \text{$y_3\sim x_3$\: in $\:G_3$ and
$x_1=x_2=r$,}\\
\vdots & \text{\vdots}\\
(y_N, \ldots, x_3,x_2,x_1); & \text{$y_N \sim x_N$\: in $\:G_N$ and
$x_1=x_2= \cdots = x_{N-1}=r$\,.}
\end{array}
\right.$$}
This product has plenty of applications in computer science. We first bound its dominator chromatic number. 

\begin{theorem}
If $\{G_i = (V_i, E_i) \}_{i=1}^N$ is a family of graphs (with a root vertex), then
$$\chi_d(G_N \sqcap \cdots \sqcap G_2\sqcap G_1)\leq \chi_d(G_1)\prod_{i=2}^N n(G_i)\,.$$
\end{theorem}

\begin{proof}
Let $c$ be a dominator coloring of $G_1$ using $\chi_d(G_1)$ colors. Set  
 $H=G_N \sqcap \cdots\sqcap G_2\sqcap G_1$ and define a coloring $f$ of $H$ with $f(x_{i_N},\ldots ,x_{i_2},x_{i_1})=(i_N,\ldots ,i_2,c(x_{i_1}))$ for $(x_{i_N},\ldots ,x_{i_2},x_{i_1})\in V(H)$. 

If $(x_{i_N},\ldots ,x_{i_2},x_{i_1})(x_{j_N},\ldots ,x_{j_2},x_{j_1})$ is an edge of 
$H$, then there exist $k\in \{i_1, \ldots, i_N\}$ and $l\in \{j_1, \ldots, j_N\}$ such that $x_kx_l\in \cup_{i=1}^NE(G_i)$.
Either way, $f(x_{i_N},\ldots ,x_{i_2},x_{i_1})\neq f(x_{j_N},\ldots ,x_{j_2},x_{j_1})$ 
and so $f$ is a proper coloring of 
$H$ with $\chi_d(G_1)\prod_{i=2}^N n(G_i)$ colors. 

It remains to prove that $f$ is a dominator coloring. It suffices to show
that each vertex of $H$ dominates at least one color class. Denote the color classes of $G$ corresponding to $c$ briefly with $C_i = C_{G_1}(i)$, $i\in [\chi_d(G_1)]$. 
Then by definition of $f$, the set $V_{ij}=\{(x_{i_N},\ldots,x_{i_2},x_{i_1}) \; | \; x_{i_1}\in V_j\}$, where $i\in\left[ \prod_{i=2}^N n(G_i)\right]$ and $j\in[\chi_d(G_1)]$, is a color class of $H$ with respect to $f$. 
Consider a vertex $(x_{i_N},\ldots ,x_{i_2},x_{i_1})\in V(H)$. Since $c$ is 
a dominator coloring of $G_1$, there exists a color class $C_j$ which is dominated by 
$x_{i_1}$. Therefore, the color class $V_{ij}$ is dominated by $(x_{i_N},\ldots ,x_{i_2},x_{i_1})$ and we are done.
\end{proof}

Note that the graph $H\sqcap G$ is obtained from $n(G)$ copies of $H$ and one copy of $G$.
In the following we will use $H_i$ to denote the copies of $H$, and $G'$ to denote the copy of $G$ in $H\sqcap G$. Also, $r_i$ will be the root vertex of $H_i$.

\begin{theorem}
If $G$ is a graph and $H$ a rooted graph with the root $r$, then
$$\chi_{dom}(H\sqcap G)=\begin{cases}
n(G)\chi_{dom}(H); & \chi_{dom}(H)=\chi_{dom}(H-r),\\
n(G)\chi_{dom}(H-r) + I(r)\chi_{dom}(G); & \text{otherwise},
\end{cases}$$
where $I(r) = 0$ if there exists an optimal dominated coloring of $H$ such that $r$ is adjacent to all vertices of at least one color class, otherwise $I(r) = 1$.
\end{theorem}

\begin{proof}
It is straightforward to see that $\chi_{dom}(H\sqcap G)$ is at most the claimed expressions, hence it remains to prove that 
$$\chi_{dom}(H\sqcap G)\geq\begin{cases}
n(G)\chi_{dom}(H); & \chi_{dom}(H)=\chi_{dom}(H-r),\\
n(G)\chi_{dom}(H-r) + I(r)\chi_{dom}(G); & \text{otherwise}\,.
\end{cases}$$
By definition of the dominated coloring, there does not exist
vertices from $V(H_i-r_i)$ and $V(H_j-r_j)$, where $i\neq j$, with the same color.
Hence at least $n(G)\chi_{dom}(H-r)$ different colors are needed in a dominated coloring of the subgraphs $H_i$ and so $\chi_{dom}(H\sqcap G)\geq |V(G)|\chi_{dom}(H-r)$. So, if $\chi_{dom}(H)=\chi_{dom}(H-r)$, then $\chi_{dom}(H\sqcap G)\geq n(G)\chi_{dom}(H-r) = n(G)\chi_{dom}(H)$.
Also, in the case that in some optimal dominated coloring $r$ is adjacent to all vertices of at least one color class in $H$, we can assign the color of the class which is dominated with $r_i$ to $r_j$ if $r_ir_j\in E(H\sqcap G)$, and so $H\sqcap G$ could be colored with $n(G)\chi_{dom}(H-r)$ different colors. Otherwise, we need at least $\chi_{dom}(G)$ different colors for dominated coloring of vertices of $G'$, and so $\chi_{dom}(H\sqcap G)\geq \chi_{dom}(G)+n(G)\chi_{dom}(H-r) $.
\end{proof}

Let $G_1,\ldots,G_k$ be rooted graphs with root vertices $r_1,\ldots,r_k$, respectively. The bridge-cycle graph $BC(G_1,\ldots,G_k; r_1,\ldots,r_k)$ is the graph obtained from the graphs $G_1,\ldots,G_k$ by joining the vertices $r_i$ and $r_{i+1}$ for all $i\in [r-1]$ and connecting the vertices $r_1$ and $r_k$ by an edge, see Fig.~\ref{fig3}. 

\begin{figure}[ht!] 
\centerline{\includegraphics[scale=.3]{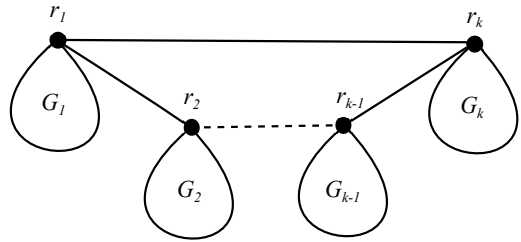}}
\caption{\label{fig3} The bridge-cycle graph $BC(G_1,\ldots,G_k; r_1,\ldots,r_k)$.}
\end{figure}
If $G_1=\cdots=G_k=G$, then we have $BC(G_1,\ldots,G_k; r_1,\ldots,r_k)\cong G\sqcap C_k$. Combining the fact that $\chi_{dom}(C_k)=\begin{cases}
\frac{k}{2} & \text{if} \; 4\mid k,\\
\lfloor\frac{k}{2}\rfloor+1 & \text{otherwise},
\end{cases}$, see~\cite{TJM}, and Theorem 2.8, we obtain that 
$\chi_{dom}(BC(G,\ldots,G; r,\ldots,r))=\chi_{dom}(G\sqcap C_k)$. Consequently,  
{\small $$\chi_{dom}(G\sqcap C_k)=\begin{cases}
k \chi_{dom}(G); & \chi_{dom}(G)=\chi_{dom}(G-r)\,,\\
\frac{k I(r)}{2}+k\chi_{dom}(G-r); & \chi_{dom}(G)\neq \chi_{dom}(G-r) \ \text{and}\ 4\mid k\,,\\
(\lfloor\frac{k}{2}\rfloor+1) I(r)+k\chi_{dom}(G-r); & \chi_{dom}(G)\neq \chi_{dom}(G-r) \ \text{and}\ 4\nmid k\,.
\end{cases}$$}

\section*{Acknowledgements}

S.K.\ acknowledges the financial support from the Slovenian Research Agency (research core funding No.\ P1-0297 and projects J1-9109, j1-1693, N1-0095, N1-0108).

\end{document}